\documentclass{scrartcl}

\usepackage{amsmath,amssymb}
\usepackage{amsthm}
\usepackage{graphicx}
\usepackage{subfigure}

\usepackage[utf8]{inputenc}
\usepackage[T1]{fontenc}
\usepackage{lmodern}

\usepackage{tikz}
\usepackage[arrow, matrix, curve]{xy}

\usepackage{hyperref}

\numberwithin{equation}{section}

\newtheorem{theorem}{Theorem}[section]

\newtheorem{lemma}[theorem]{Lemma}

\theoremstyle{definition}

\newtheorem{remark}[theorem]{Remark}

\newtheorem*{theorem*}{Theorem}

\newtheorem*{thma}{Theorem A}
\newtheorem*{thmb}{Theorem B}

\begin{document}
\title{The Spinorial Energy Functional: Solutions of the Gradient Flow on Berger Spheres}
\author{Johannes Wittmann}
\maketitle

\begin{abstract}
We study the negative gradient flow of the spinorial energy functional (introduced by Ammann, Weiß, and Witt) on 3-dimensional Berger spheres. For a certain class of spinors we show that the Berger spheres collapse to a 2-dimensional sphere. Moreover, for special cases, we prove that the volume-normalized standard 3-sphere together with a Killing spinor is a stable critical point of the volume-normalized version of the flow. Our results also include an example of a critical point of the volume-normalized flow on the 3-sphere, which is not a Killing spinor.
\end{abstract}
\section{Introduction}
Let $M$ be a compact spin manifold and $\mathcal{N}$ the union of all pairs $(g,\varphi)$ where $g$ is a Riemannian metric on $M$ and $\varphi\in\Gamma(\Sigma(M,g))$ is a spinor of the spin manifold $(M,g)$ whose pointwise norm is constant and equal to $1$. The \textit{spinorial energy functional} $\mathcal{E}$, introduced in \cite{AWWsflow}, is defined by
\[\mathcal{E}\colon \mathcal{N}\rightarrow [0,\infty), \hspace{3em}(g,\varphi)\mapsto \frac{1}{2}\int_M|\nabla^{\Sigma(M,g)}\varphi|^2dv^g,\]
where $dv^g$ is the Riemannian volume form of $(M,g)$ and $|.|$ is the pointwise norm on $T^*M\otimes \Sigma(M,g)$. If $\textup{dim}M\ge 3$, then the critical points of $\mathcal{E}$ are precisely the pairs $(g,\varphi)$ consisting of a Ricci-flat Riemannian metric $g$ and a parallel spinor $\varphi$. In the surface case, the spinorial energy functional is related to the Willmore energy of immersions and treated in detail in \cite{AWWsflowsurf}.

On the Fréchet-bundle $\mathcal{N}\rightarrow\mathcal M$, $\mathcal{M}:=\{\text{Riemannian metrics on } M\}$, there exists a natural connection, which is defined in \cite{AWWsflow} with the aid of results in \cite{BourGaud}. This connection defines a splitting of $T\mathcal{N}$ in horizontal and vertical subbundles, which allows us to define a Riemannian metric on $\mathcal{N}$. The negative gradient flow of $\mathcal{E}$ with respect to this Riemannian metric is called the \textit{spinor flow}. Short time existence and uniqueness of the spinor flow was shown in \cite{AWWsflow} with a variant of DeTurck's trick.

In this paper the spinor flow on 3-dimensional Berger spheres is treated. We view the 3-sphere $S^3$ as a $S^1$-principal bundle over $S^2$ via the Hopf fibration $\pi\colon S^3\rightarrow S^2$. Rescaling the standard metric $g_{S^3}$ along the fibers of the Hopf fibration by $\varepsilon >0$ yields the Berger metrics $g^\varepsilon$ on $S^3$. We call $(S^3,g^\varepsilon)$ a Berger sphere.

There is a certain class of spinors on $S^3$, the so-called $S^1$-invariant spinors \cite{ABCollaps}, \cite{Mor}, which are in one-to-one correspondence to the spinors on the base manifold $S^2$. Our first theorem concerns these spinors.

\begin{thma}[Collapse] Let $M=S^3$ and as initial value $(g_0,\varphi_0)$ choose $g_0=g^\varepsilon$ and $\varphi_0$ a spinor of unit length that corresponds to an arbitrary Killing spinor on the base $S^2$. Then, if the fibers are sufficiently short (i.e. $\varepsilon$ is small enough), the spinor flow converges to a 2-dimensional sphere in infinite time. 
\end{thma}
This theorem can be seen as a special case of the conjecture that $S^1$-principal bundles with suitable Riemannian metrics and sufficiently short fibers together with $S^1$-invariant spinors collapse to the base manifold under the spinor flow.

In \cite{AWWsflow} it was observed that the volume-normalized standard metric on $S^3$ together with a Killing spinor is a critical point of the volume-normalized spinor flow. It is not clear whether this critical point is stable. However, there are such stability results for other geometric flows, see e.g. \cite[1.1 Theorem]{Huisken} in the case of the mean curvature flow. Our second theorem is a first positive result concerning this stability question.
 
\begin{thmb}[Stability] Let $M=S^3$ and as initial value $(g_0,\varphi_0)$ choose $g_0=c(\varepsilon)g^\varepsilon$ the volume-normalized Berger metric and $\varphi_0$ a spinor that is obtained via parallel transport of an arbitrary Killing spinor of unit length from $(S^3,g_{S^3})$ to $(S^3,g^\varepsilon)$ as described in Remark \ref{rmk main res 2}. Then, if we are not too far away from $c(1)g_{S^3}$ (i.e. $\varepsilon$ is sufficiently close to $1$), the volume-normalized spinor flow converges in infinite time to the volume-normalized standard metric on $S^3$ together with a Killing spinor.
\end{thmb}

\subsection{Overview of the proof}
First of all, in \cite{AWWsflow} it was shown that under the splitting of $T\mathcal{N}$ the negative gradient of $\mathcal{E}$ has an expression
\[-\textup{grad}\mathcal{E}_{(g,\varphi)}=(Q_1(g,\varphi),Q_2(g,\varphi))\in\Gamma(\odot^2T^*M)\oplus\Gamma(\Sigma(M,g)),\]
where $Q_1(g,\varphi)$ and $Q_2(g,\varphi)$ depend mainly on $\nabla^{(M,g)}\varphi$. This fact is important for us, because it means, essentially, that we can understand $-\textup{grad}\mathcal{E}_{(g,\varphi)}$ by understanding $\nabla^{(M,g)}\varphi$.

Furthermore, one of the main tools for us to prove the above theorems are generalized cylinders \cite{BaerGaudMoroi}, which provide a way to identify spinors for different metrics. To be more concrete, given a smooth $1$-parameter family $(g_t)_{t\in I}$ of Riemannian metrics on a manifold $M$, $I$ an interval, the generalized cylinder is the manifold $\mathcal{Z}:=I\times M$ together with the Riemannian metric $g_\mathcal{Z}:=dt^2+g_t$. If the dimension of $M$ is odd, as in our case, we get an identification $\Sigma^+(\mathcal Z,g_\mathcal{Z})|_{\{t\}\times M}\cong\Sigma (M,g_t)$. In particular, we can think of sections $\varphi\in\Gamma(\Sigma^+(\mathcal Z,g_\mathcal{Z}))$ as families of sections $(\varphi_t)_{t\in I}$ with $\varphi_t\in\Gamma(\Sigma(M,g_t))$ where $\varphi_t(.):=\varphi(t,.)$.
 
Denote by $\pi\colon S^3\rightarrow S^2$ the Hopf fibration as above. We write
\begin{align}\label{eq ansatz for g_t}
g_t(X_1+Y_1,X_2+Y_2):=g_{S^3}(\alpha(t)X_1+\beta(t)Y_1, \alpha(t)X_2+\beta(t)Y_2)
\end{align}
for $t\in I=[0,b)$, $b\in(0,\infty]$, $X_i\in \textup{ker}(d\pi)$, $Y_i\in\textup{ker}(d\pi)^\bot$, $i=1,2$, and smooth functions $\alpha,$ $\beta\colon I\rightarrow (0,\infty)$. We will choose \eqref{eq ansatz for g_t} as ansatz for the metric part of the solution where we require that $\alpha(0)$ and $\beta(0)$ are chosen so that $g_0$ is the metric part of our initial value.

Using the generalized cylinder with respect to \eqref{eq ansatz for g_t} we write
\begin{align}\label{eq ansatz for varphi_t}
\varphi_t(p):=\mathcal{P}_{0,t}(p)(\varphi_0(p))
\end{align}
where $\mathcal{P}_{0,t}(p)$ is the parallel transport in $\Sigma^+(\mathcal{Z},g_\mathcal{Z})$ with respect to $\nabla^{\Sigma^+(\mathcal{Z},g_{\mathcal{Z}})}$ along the curve $\gamma_p(s):=(s,p)$ from $\gamma_p(0)$ to $\gamma_p(t)$. Then we choose \eqref{eq ansatz for varphi_t} as ansatz for the spinor part of the solution. In the next step, we derive an expression for $\nabla^{\Sigma(S^3,g_t)}\varphi_t$ that depends in particular on $\alpha$ and $\beta$. To achieve this, we use curvature terms to construct suitable differential equations in $\Sigma^+(\mathcal{Z},g_\mathcal{Z})$. We use these expressions for $\nabla^{\Sigma(S^3,g_t)}\varphi_t$ to calculate $Q_1(g_t,\varphi_t)$ and $Q_2(g_t,\varphi_t)$. After that we show $\frac{\partial}{\partial t}\varphi_t=0=Q_2(g_t,\varphi_t)$ independent of the choice of $\alpha$ and $\beta$. Finally, we will see that $\frac{\partial}{\partial t}g_t=Q_1(g_t,\varphi_t)$ is equivalent to a system of two non-linear ordinary differential equations for $\alpha$ and $\beta$. We solve these systems to get the desired properties of the solutions.

\subsection*{Acknowledgments} I would like to thank Bernd Ammann for his ongoing support and many fruitful discussions. I am also grateful to Nicolas Ginoux for his insightful comments at the early stages of my research.

\section{Preliminaries}
\subsection{Spin geometry} In this section we fix notation and review basics of spin geometry which will be relevant in the following. For more details we refer to e.g. \cite{LaMi}, \cite{Hij}, \cite{Fri} and \cite{Roe}.

Let $M$ be an oriented $n$-dimensional manifold and denote by $\textup{GL}^+M$ the $GL^+(n,\mathbb{R})$-principal bundle of oriented frames for $M$. Moreover, we denote by $\theta\colon \widetilde{GL}^+(n,\mathbb{R})\rightarrow GL^+(n,\mathbb{R})$ the universal covering for $n\ge 3$ and the connected twofold covering for $n=2$. A \textit{topological spin structure on $M$} is a $\theta$-reduction of $\textup{GL}^+M$, i.e. a topological spin structure on $M$ is a $\widetilde{GL}^+(n,\mathbb{R})$-principal bundle $\widetilde{\textup{GL}}^+M$ over $M$ together with a twofold covering $\Theta\colon\widetilde{\textup{GL}}^+M\rightarrow\textup{GL}^+M$ such that the following diagram commutes
\[
\begin{xy}
\xymatrix{
	\widetilde{\textup{GL}}^+M\times \widetilde{GL}^+(n,\mathbb{R}) \ar[r] \ar[dd]^{\Theta\times\theta}  &   \widetilde{\textup{GL}}^+M \ar[dd]^\Theta  \ar[rd] & \\
	& & M\\
	\textup{GL}^+M\times GL^+(n,\mathbb{R}) \ar[r]             &   \textup{GL}^+M \ar[ru] & 
}
\end{xy}
\]
where the horizontal arrows denote the group actions of the principal bundles. Now let $(M,g)$ be an oriented Riemannian manifold and $\textup{SO}(M,g)$ the $SO(n,\mathbb{R})$-principal bundle of oriented orthonormal frames for $M$. Restricting $\theta$ to the \textit{spin group} given by $\textup{Spin}(n):=\theta^{-1}(SO(n,\mathbb{R}))$, we define a \textit{metric spin structure on $M$} to be a $\theta|_{\textup{Spin}(n)}$-reduction of $\textup{SO}(M,g)$. Again, this means that a metric spin structure on $M$ is a $\textup{Spin}(n)$-principal bundle $\textup{Spin}(M,g)$ over $M$ together with a twofold covering $\Theta\colon \textup{Spin}(M,g)\rightarrow \textup{SO}(M,g)$ such that the following diagram commutes
\[
\begin{xy}
\xymatrix{
	\textup{Spin}(M,g)\times \textup{Spin}(n) \ar[r] \ar[dd]^{\Theta\times\theta}  &   \textup{Spin}(M,g) \ar[dd]^\Theta  \ar[rd] & \\
	& & M\\
	\textup{SO}(M,g)\times SO(n,\mathbb{R}) \ar[r]             &   \textup{SO}(M,g) \ar[ru] & 
}
\end{xy}
\]
Given a topological spin structure $\widetilde{\textup{GL}}^+M$ on an oriented manifold $M$, every Riemannian metric $g$ on $M$ defines a metric spin structure on $(M,g)$ by $\textup{Spin}(M,g):=\widetilde{\textup{GL}}^+M|_{\textup{SO}(M,g)}$. In the following, the term \textit{spin structure} refers to a topological or metric spin structure and it should always be clear from the context which one we mean. 

In order to introduce (complex) spinors, we consider representations of $\textup{Spin}(n)$. We first note that the spin group can be realized as a subgroup of the group of invertible elements in $\mathbb{C}l_n$ where $\mathbb{C}l_n$ is the Clifford algebra of $\mathbb{C}^n$ with inner product given by the complex bilinear extension of the standard inner product of $\mathbb{R}^n$, namely $\textup{Spin}(n)\cong\{x_1\cdot\ldots\cdot x_{2k}\text{ }|\text{ } x_i\in S^{n-1}\subset\mathbb{R}^n\subset\mathbb{C}l_n,\text{ }k\in\mathbb{N}\}$. If $n$ is even, then there exists exactly one equivalence class of irreducible complex representations of $\mathbb{C}l_n$ and every such representation is of dimension $2^\frac{n}{2}$. If $n$ is odd, then there exist exactly two equivalence classes of irreducible complex representations of $\mathbb{C}l_n$, each of dimension $2^{\frac{n-1}{2}}$. Introducing the \textit{complex volume element} $\omega_n:=i^{\lfloor\frac{n+1}{2}\rfloor} e_1\cdot\ldots\cdot e_n\in\mathbb{C}l_n$ where $a=\lfloor b\rfloor$ is the largest integer $a\le b$ and $(e_1,\ldots,e_n)$ is the standard basis of $\mathbb{C}^n$, we can distinguish the two different equivalence classes for $n$ odd by the action of $\omega_n$, i.e. $\omega_n$ acts as the identity $id$ on one equivalence class and as $-id$ on the other. The \textit{complex spinor representation} $\rho\colon \textup{Spin}(n)\rightarrow \textup{Aut}(\Sigma_n)$ is the restriction of an irreducible complex representation $\rho\colon \mathbb{C}l_n\rightarrow \textup{End}(\Sigma_n)$ of $\mathbb{C}l_n$ to $\textup{Spin}(n)$ where for $n$ odd we require $\rho(\omega_n)=id_{\Sigma_n}$. For $n$ odd, the complex spinor representation is irreducible. For $n$ even, it splits into two irreducible representations $\rho=\rho^+\oplus\rho^-$ where $\rho^\pm\colon \textup{Spin}(n)\rightarrow \textup{Aut}(\Sigma_n^\pm)$ have dimension $2^{\frac{n}{2}-1}$ and $\Sigma_n^\pm$ are the $\pm1$-eigenspaces of $\rho(\omega_n)$.

Let $\textup{Spin}(M,g)$ be a spin structure on $(M,g)$. The \textit{(complex) spinor bundle} $\Sigma (M,g)$ is the complex vector bundle associated to the spin structure and the  complex spinor representation, i.e. $\Sigma(M,g):=\textup{Spin}(M,g)\times_\rho \Sigma_n$. For $n$ even, we have an isomorphism $\Sigma (M,g)\cong \Sigma^+ (M,g) \oplus\Sigma^-(M,g)$ where $\Sigma^\pm(M,g):=\textup{Spin}(M,g)\times_{\rho^\pm}\Sigma_n^\pm$. Next we introduce the so-called Clifford multiplication, which allows to multiply spinors and tangent vectors. To that end, notice that $TM\cong \textup{Spin}(M,g)\times_{\tau\circ\theta|_{\textup{Spin(n)}}}\mathbb{R}^n$, where $\tau$ is the standard representation of $SO(n,\mathbb{R})$ on $\mathbb{R}^n$. Given $\varphi=[p,\sigma]\in\Sigma_x(M,g)$ and $X=[p,v]\in T_xM$ we define the \textit{Clifford multiplication (on $\Sigma (M,g)$)} by $X\cdot\varphi:=[p,\rho(v)(\sigma)]$. From the relations of the Clifford algebra $\mathbb{C}l_n$, it follows that
\begin{align}\label{eq spingeo cliffordm}
	X\cdot (Y\cdot \varphi) + Y\cdot (X\cdot \varphi)=-2g(X,Y)\varphi,
\end{align} 
for all $X,Y\in T_xM$ and $\varphi\in\Sigma_x(M,g)$. For $n$ even, Clifford multiplication interchanges the factors $\Sigma^\pm(M,g)$. Moreover, given an oriented orthonormal basis $(e_1,\ldots,e_n)$ of $T_xM$ and $\varphi\in\Sigma_x(M,g)$ for $n$ odd, respectively $\varphi\in\Sigma^+_x(M,g)$ for $n$ even, we have
\begin{align}\label{eq spingeo cliffordm onb}
	i^{\lfloor\frac{n+1}{2}\rfloor} e_1\cdot\ldots\cdot e_n\cdot\varphi=\varphi.
\end{align}
To measure the length of spinors, we introduce a natural bundle metric on $\Sigma(M,g)$. First, given an irreducible representation $\rho\colon\mathbb{C}l_n\rightarrow \textup{End}(\Sigma_n)$ of $\mathbb{C}l_n$, there exists a hermitian inner product $\langle.,.\rangle_{\Sigma_n}$ on $\Sigma_n$ such that $\langle\rho(x)(\psi),\varphi\rangle_{\Sigma_n}=\langle\psi,\rho(x)(\varphi)\rangle_{\Sigma_n}$ for all $x\in\mathbb{R}^n$, $\varphi$, $\psi\in\Sigma_n$. In particular, the inner product $\langle.,.\rangle_{\Sigma_n}$ is $\textup{Spin}(n)$-invariant and therefore induces a bundle metric on $\Sigma(M,g)$, which we denote by $\langle.,.\rangle$. It holds that
\begin{align}\label{eq spingeo cliffordm scp}
	\langle X\cdot\psi,\varphi\rangle=-\langle\psi,X\cdot\varphi\rangle,
\end{align}
for all $X\in T_xM$, $\varphi$, $\psi\in\Sigma_x(M,g)$. For $\varphi\in\Gamma(\Sigma(M,g))$ we set $|\varphi|:=\sqrt{\langle\varphi,\varphi\rangle}$. In order to differentiate spinors we note that the Levi-Civita connection $\nabla$ on $(M,g)$ can be lifted to a metric connection $\nabla^{\Sigma(M,g)}$ on $\Sigma(M,g)$, the \textit{spinorial Levi-Civita connection}. For all $X$, $Y\in\Gamma(TM)$ and $\varphi\in\Gamma(\Sigma(M,g))$ we have 
\begin{align*}
	\nabla^{\Sigma(M,g)}_X(Y\cdot\varphi)=(\nabla_XY)\cdot\varphi + Y\cdot \nabla^{\Sigma(M,g)}_X\varphi.
\end{align*}
For $n$ even, the factors $\Sigma^\pm(M,g)$ are invariant under $\nabla^{\Sigma(M,g)}$. In particular, we get connections $\nabla^{\Sigma^\pm(M,g)}$ on $\Sigma^\pm(M,g)$. Denote by $R^{\Sigma(M,g)}$ the curvature of $\Sigma(M,g)$. Let $(e_1,\ldots,e_n)$ be a local orthonormal frame for $(M,g)$. It holds that
\begin{align}\label{eq spingeo curvature}
	R^{\Sigma(M,g)}(X,Y)\varphi=\frac12\sum_{1\le i < j\le n} g(R^M(X,Y)e_i,e_j)e_i\cdot(e_j\cdot \varphi),
\end{align}  
for all $X$, $Y\in\Gamma(TM)$, $\varphi\in\Gamma(\Sigma(M,g))$ where $R^M$ is the curvature of $(M,g)$.

\subsection{Generalized cylinders}
Details concerning this section can be found in \cite{BaerGaudMoroi}. Let $M$ be a manifold, $I\subset\mathbb{R}$ an interval and $(g_t)_{t\in I}$ a smooth 1-parameter family of Riemannian metrics on $M$. The \textit{generalized cylinder} is the Riemannian manifold $(\mathcal{Z},g_\mathcal{Z})$, where $\mathcal{Z}:=I\times M$ and $g_\mathcal{Z}:=dt^2+g_t$.
The Riemannian hypersurface $\{t\}\times M$ is isometric to $(M,g_t)$ and we denote both by $M_t$. Moreover, the vector field $\nu:=\frac{\partial}{\partial t}\in\Gamma(T\mathcal{Z})$ is of unit length and $\nu|{M_t}$ is orthogonal to $M_t$. We write $W=W_t$ for the Weingarten map of $M_t$ with respect to $\nu|_{M_t}$.

The following lemma will be used later.
\begin{lemma}
	For all $U$, $X$, $Y\in T_pM$, $p\in M$, and $t\in I$, it holds that
	\begin{align}
	\nabla^{\mathcal{Z}}_\nu \nu &=0,\notag \\
	g_t(W_t(X),Y)&=-\frac{1}{2}\dot{g}_t(X,Y),\notag \\
	g_\mathcal{Z}(R^\mathcal{Z}(X,Y)U,\nu)&=\frac{1}{2}((\nabla^{M_t}_Y\dot{g}_t)(X,U)-(\nabla^{M_t}_X\dot{g}_t)(Y,U)),\label{eq cy 3}\\
	g_\mathcal{Z}(R^\mathcal{Z}(X,\nu)\nu,Y)&=-\frac{1}{2}(\ddot{g}_t(X,Y)+\dot{g}_t(W_t(X),Y))\label{eq cy 4}.
	\end{align}
	If $\tilde{Z}\in\Gamma(T\mathcal{Z})$ with $\tilde{Z}(t,p)=(0,Z_t(p))\in T_tI\times T_pM$ for all $t\in I$, $p\in M$, then
	\begin{align}
	[\nu,\tilde{Z}](t,p)&=(0,\frac{d}{ds}\bigg|_{s=t}Z_s(p))\in T_tI\times T_pM,\label{eq cy 5}\\
	(\nabla^\mathcal{Z}_\nu \tilde{Z})(t,p)&=(0,\frac{d}{ds}\bigg|_{s=t}Z_s(p)-W_t(Z_t(p)))\label{eq cy 6}
	\end{align}
	for all $t\in I$, $p\in M$ where $\nabla^{\mathcal{Z}}$ is the Levi-Civita connection of $(\mathcal{Z},g_\mathcal{Z})$.
\end{lemma}

The next lemma describes how we can identify spinors of different spinor bundles with the help of generalized cylinders.

\begin{lemma}\label{lemma cyl identi} Let $M$ be an oriented manifold together with a topological spin structure $\widetilde{\textup{GL}}^+M$. The topological spin structure on $M$  induces a metric spin structure on $(\mathcal{Z}, g_\mathcal{Z})$ and metric spin structures $\textup{Spin}(M,g_t):=\widetilde{\textup{GL}}^+M|_{\textup{SO}(M,g_t)}$ on $(M,g_t)$. For the respective spinor bundles we have the following isomorphisms of vector bundles: If $n$ is even, then $\Sigma(\mathcal Z,g_\mathcal{Z})|_{M_t}\cong\Sigma M_t$. If $n$ is odd, then $\Sigma^+(\mathcal Z,g_\mathcal{Z})|_{M_t}\cong\Sigma M_t$. The bundle metrics $\langle.,.\rangle$ are preserved by these isomorphisms. Moreover, if ``$\cdot$'' and ``$\cdot_t$'' denote the Clifford multiplications in $\Sigma^{(+)}(\mathcal{Z},g_\mathcal{Z})$ and $\Sigma M_t$, then it holds that 
\[\nu\cdot(X\cdot\varphi)=X\cdot_t\varphi\]
for all $X\in TM$, $\varphi\in\Sigma M_t$. If we write $\nabla^t=\nabla^{\Sigma M_t}$, then we have
\[\nabla^{\Sigma(\mathcal{Z},g_\mathcal{Z})}_X\varphi=\nabla^t_X\varphi -\frac{1}{2}W_t(X)\cdot_t\varphi\]
for all $\varphi\in\Gamma(\Sigma M_t)$.
\end{lemma}

\subsection{The spinorial energy functional and its gradient flow}\label{sect sflow}
In the following we work with the real part of $\langle.,.\rangle$ and we write $(.,.):=\textup{Re}\langle.,.\rangle$. It will be useful that $(X\cdot\varphi,Y\cdot\varphi)=g(X,Y)(\varphi,\varphi)$ and $(X\cdot\varphi,\varphi)=0$ hold for all $X$, $Y\in T_xM$ and $\varphi\in\Sigma_x(M,g)$. These identities follow directly from \eqref{eq spingeo cliffordm} and \eqref{eq spingeo cliffordm scp}.

Let $M$ be a connected, compact, oriented manifold with a fixed topological spin structure $\widetilde{\textup{GL}}^+M$ and $\textup{dim}M\ge 2$. As stated before, every choice of Riemannian metric $g$ on $M$ defines a metric spin structure $\textup{Spin}(M,g):=\widetilde{\textup{GL}}^+M|_{\textup{SO}(M,g)}$ and so we have the corresponding spinor bundles $\Sigma(M,g)$. We set $\mathcal{N}_g:=\{\varphi\in\Gamma(\Sigma(M,g))\text{ }|\text{ }|\varphi|=1\}$, and $\mathcal{N}:=\bigsqcup_{g\in\mathcal{M}}\mathcal{N}_g$. The \textit{spinorial energy functional} $\mathcal{E}$ is defined by
\[\mathcal{E}\colon \mathcal{N}\rightarrow [0,\infty), \hspace{3em}(g,\varphi)\mapsto \frac{1}{2}\int_M|\nabla^{\Sigma(M,g)}\varphi|^2dv^g,\]
As mentioned in the introduction there exists a natural connection on the Fréchet-bundle $\mathcal{N}\rightarrow\mathcal M$. For details we refer to \cite{AWWsflow}. From that connection we get horizontal tangent spaces $\mathcal{H}_{(g,\varphi)}\cong\Gamma(\odot^2T^*M)$ and a splitting
\begin{align} \label{eq z2}
T_{(g,\varphi)}\mathcal{N}\cong\Gamma(\odot^2T^*M)\oplus V_{(g,\varphi)}
\end{align}
where
\[V_{(g,\varphi)}=\{\psi\in\Gamma(\Sigma(M,g))\text{ }|\text{ }(\varphi,\psi)=0 \}.\]
On the first factor, we choose the inner product which we get by integrating the natural inner product on $(2,0)$-tensors. On the second factor, we choose the $L^2$-inner product defined by
\[(\psi_1,\psi_2)_{L^2}:=\int_M(\psi_1,\psi_2)dv^g,\]
for $\psi_1$, $\psi_2\in V_{(g,\varphi)}$. The negative gradient flow of $\mathcal{E}$,
\[\frac{\partial}{\partial t}(g_t,\varphi_t)=-\textup{grad}\mathcal{E}_{(g_t,\varphi_t)},\]
is called the \textit{spinor flow}. Under the splitting \eqref{eq z2} we have 
\[-\textup{grad}\mathcal{E}_{(g,\varphi)}=(Q_1(g,\varphi),Q_2(g,\varphi)),\] 
with
\begin{align*}
Q_1(g,\varphi)&=-\frac{1}{4}|\nabla^{\Sigma(M,g)}\varphi|^2g-\frac{1}{4}\textup{div}_g T_{g,\varphi}+\frac{1}{2}\left(\nabla^{\Sigma(M,g)}\varphi\otimes\nabla^{\Sigma(M,g)}\varphi\right),\\
Q_2(g,\varphi)&=-\left(\left({\nabla^{\Sigma (M,g)}}\right)^*\nabla^{\Sigma (M,g)}\right)\varphi+|\nabla^{\Sigma(M,g)}\varphi|^2\varphi,
\end{align*}
for all $(g,\varphi)\in\mathcal{N}$. \begin{sloppypar}
Here, $\left(\nabla^{\Sigma(M,g)}\varphi\otimes\nabla^{\Sigma(M,g)}\varphi\right)(X,Y):=\left(\nabla^{\Sigma(M,g)}_X\varphi,\nabla^{\Sigma(M,g)}_Y\varphi\right)$
for all $X,Y\in\Gamma(TM)$, and $T_{g,\varphi}$ is the symmetrization of $\left(X\cdot Y\cdot\varphi + g(X,Y)\varphi,\nabla^{\Sigma(M,g)}_Z\varphi\right)$ in the second and third component where $X,Y,Z\in\Gamma(TM)$.\end{sloppypar}
The \textit{(volume) normalized spinor flow} is the negative gradient flow of $\mathcal{E}|_{\mathcal{N}_1}$ for $\mathcal{N}_1:=\{(g,\varphi)\in\mathcal{N}\text{ }|\text{ } \textup{vol}(M,g)=1\}$. We have $-\textup{grad}(\mathcal{E}|_{\mathcal{N}_1})=(\tilde{Q}_1,Q_2|_{\mathcal{N}_1})$ with 
\[\tilde{Q}_1(g,\varphi)=Q_1(g,\varphi)+\frac{n-2}{2n}\frac{1}{\textup{vol}(M,g)}\mathcal{E}(g,\varphi)g\]
for all $(g,\varphi)\in\mathcal{N}_1$. Note that in the above identity for $\tilde{Q}_1$ the ``$\frac{1}{\textup{vol}(M,g)}$'' is equal to $1$. However, it is important in our strategy of solving the normalized spinor flow. More concretely, we will construct $(g_t,\varphi_t)$ with $|\varphi_t|=1$ for all $t$ and $\frac{\partial}{\partial t}(g_t,\varphi_t)=(\tilde{Q}_1(g_t,\varphi_t),Q_2(g_t,\varphi_t))$ where for the initial value we have $\textup{vol}(M,g_0)=1$. To make sure that $(g_t,\varphi_t)$ is a solution of the normalized spinor flow, we then need to verify that $\textup{vol}(M,g_t)=1$ for all $t$. To that end, we calculate
\begin{align*}
\frac{d}{dt}\bigg|_{t=s}\textup{vol}(M,g_t)&=\frac{d}{dt}\bigg|_{t=s}\int_M dv^{g_t}=\int_M\frac{d}{dt}\bigg|_{t=s}dv^{g_t}\\
&=\int_M (\frac{1}{2}\textup{tr}_{g_{s}}\dot{g}_{s})dv^{g_{s}}\\
&=\int_M (\frac{1}{2}\textup{tr}_{g_{s}}\tilde{Q}_1(g_s,\varphi_s))dv^{g_{s}}\\
&=\frac{1}{2}\int_M\textup{tr}_{g_{s}}Q_1(g_s,\varphi_s)dv^{g_{s}}+\frac{1}{2}\frac{n-2}{2n}\frac{1}{\textup{vol}(M,g_s)}\mathcal{E}(g_s,\varphi_s)\int_M\textup{tr}_{g_s}g_sdv^{g_s}\\
&=\frac{1}{2}\int_M-\frac{n-2}{4}|\nabla^{\Sigma(M,g_s)}\varphi_s|^2dv^{g_s}+\frac{1}{2}\frac{n-2}{2}\frac{1}{\textup{vol}(M,g_s)}\mathcal{E}(g_s,\varphi_s)\textup{vol}(M,g_s)\\
&=-\frac{1}{2}\frac{n-2}{2}\mathcal{E}(g_s,\varphi_s)+\frac{1}{2}\frac{n-2}{2}\mathcal{E}(g_s,\varphi_s)\\
&=0,
\end{align*}
so $\textup{vol}(M,g_t)=1$ for every $t$ and $(g_t,\varphi_t)$ is in fact a solution of the normalized spinor flow.

\section{Solutions of the spinor flow on Berger spheres}
In this section we state and prove our main results, Theorem \ref{main result 1} and Theorem \ref{main result 2}. First we collect necessary technical ingredients. Then we define $S^1$-invariant spinors, which are part of the initial values of theorem \ref{main result 1}. After that we prove our main results with the strategy explained in the introduction.

As stated in the introduction we view $S^3$ as a $S^1$-principal bundle over $S^2$ via the Hopf fibration $\pi\colon S^3\rightarrow S^2$. If we equip $S^2=\mathbb{C}P^1$ with the Fubini-Study metric $g_{FS}$, then the Hopf fibration $\pi\colon (S^3,g_{S^3})\rightarrow (S^2,g_{FS})$ turns into a Riemannian submersion.  The action of $S^1$ on $S^3$ induces a global flow whose infinitesimal generator we denote by $K\in\Gamma(TS^3)$. We have that $K(p)\in\textup{ker}(d\pi_p)$ for all $p\in S^3$ and $K$ is of unit length with respect to $g_{S^3}$. On $S^3$ we choose the connection which is given by
\[p\mapsto \textup{ker}(d\pi_p)^\bot.\]
This connection induces a connection form $\tilde{\omega}\colon TS^3\rightarrow i\mathbb{R}$. It holds that $\tilde{\omega}(K)= i$. We write $\omega:=\frac{1}{i}\tilde{\omega}$ and denote by $d\omega$ the differential of $\omega$, i.e. 
\[d\omega(X,Y)=L_X(\omega(Y))-L_Y(\omega(X))-\omega([X,Y]).\]
With $X^*$ we denote the horizontal lift of $X$ (with respect to the above connection).
\begin{remark}[Orientation convention for $S^3$] For the rest of this paper we fix an orientation on $S^2$. On $S^3$ we fix the orientation that satisfies the following: If $(f_1,f_2)$ is any oriented local orthonormal frame for $(S^2,g_{FS})$, then $(K,f_1^*,f_2^*)$ is an oriented local orthonormal frame for $(S^3,g_{S^3})$. 
\end{remark}
In Remark \ref{rem change or conv} we explain how our results change if we choose the other orientation on $S^3$ (i.e. $(-K,f_1^*,f_2^*)$ is oriented).

\begin{remark}[Notation for frames] For the rest of this paper we use the following notation: $(f_1,f_2)$ denotes an arbitrary oriented local orthonormal frame for $(S^2,g_{FS})$. Moreover, we set
\begin{align*}
(f_1(t),f_2(t))&:=(\frac{1}{\beta(t)}f_1,\frac{1}{\beta(t)}f_2),\\
(e_0(t), e_1(t), e_2(t))&:=(\frac{1}{\alpha(t)}K, f_1(t)^*,f_2(t)^*).
\end{align*}
Then $(e_0(t), e_1(t), e_2(t))$ is an oriented local orthonormal frame for $(S^3,g_t)$ where $g_t$, $\alpha(t)$, and $\beta(t)$ are defined by \eqref{eq ansatz for g_t}.
\end{remark}

We set
\[a:=d\omega(f_1^*,f_2^*)=\pm 2.\]
Note that $a$ is a constant that does not depend on the choice of the oriented local orthonormal frame $(f_1,f_2)$.

\begin{lemma}\label{lemma z1}If $\nabla^t$ denotes the Levi-Civita connection on $(S^3,g_t)$ and $(S^2,\beta(t)^2g_{FS})$ respectively, we have
\begin{align*}
\nabla^t_{e_0(t)}e_0(t)&=0,\hspace{5em} & \nabla^t_{e_1(t)}e_2(t)&=-\frac{1}{2}\frac{\alpha(t)}{\beta(t)^2}ae_0(t)+\left(\nabla^t_{f_1(t)}f_2(t)\right)^*,\\
\nabla^t_{e_0(t)}e_1(t)&=\frac{1}{2}\frac{\alpha(t)}{\beta(t)^2}ae_2(t),\hspace{3em} & \nabla^t_{e_2(t)}e_1(t)&=\frac{1}{2}\frac{\alpha(t)}{\beta(t)^2}ae_0(t)+\left(\nabla^t_{f_2(t)}f_1(t)\right)^*,\\
\nabla^t_{e_0(t)}e_2(t)&=-\frac{1}{2}\frac{\alpha(t)}{\beta(t)^2}ae_1(t),\hspace{3em} &\nabla^t_{e_2(t)}e_0(t)&=-\frac{1}{2}\frac{\alpha(t)}{\beta(t)^2}ae_1(t),\\
\nabla^t_{e_1(t)}e_0(t)&=\frac{1}{2}\frac{\alpha(t)}{\beta(t)^2}ae_2(t), \hspace{3em} &\nabla^t_{e_2(t)}e_2(t)&=\left(\nabla^t_{f_2(t)}f_2(t)\right)^*,\\
\nabla^t_{e_1(t)}e_1(t)&=\left(\nabla^t_{f_1(t)}f_1(t)\right)^*.
\end{align*}
\end{lemma}
\begin{proof} Since horizontal lifts are right invariant, it follows that
\[ [e_0(t),e_j(t)]=0 \]
on $S^3$ for $j=1,2$. Using the Koszul formula we then compute the Christoffel symbols $\tilde{\Gamma}_{ij}^k$ of $(e_0(t), e_1(t), e_2(t))$ with respect to $\nabla^t$:
\begin{align*}
-\tilde{\Gamma}_{ij}^0&=\tilde{\Gamma}_{i0}^j=\tilde{\Gamma}_{0i}^j=\frac{1}{2}\frac{\alpha(t)}{\beta(t)^2}a,\\
\tilde{\Gamma}_{00}^i&=\tilde{\Gamma}_{i0}^0=\tilde{\Gamma}_{0i}^0=\tilde{\Gamma}_{00}^0=0,\\
\tilde{\Gamma}_{ij}^k&=\Gamma_{ij}^k\circ \pi,
\end{align*}
for $i,j,k\in\{1,2\}$ where $\Gamma_{ij}^k$ are the Christoffel symbols of $(f_1(t), f_2(t))$ with respect to $\nabla^t$. The lemma now follows from an easy computation.
\end{proof}

\begin{lemma}\label{lemma curv in cyl}For all $\varphi\in\Gamma(\Sigma^+(\mathcal Z,g_\mathcal{Z}))$ and all horizontal vector fields $Y\in\Gamma(TS^3)$ it holds that
\begin{align*}
R^{\Sigma^+(\mathcal Z,g_\mathcal{Z})}(\nu,K)\varphi&=\frac{1}{2}\left(\frac{\alpha''(t)}{\alpha(t)}+\frac{\alpha'(t)}{\beta(t)^2}a-\frac{\alpha(t)\beta'(t)}{\beta(t)^3}a\right)\nu\cdot K\cdot\varphi,\\
R^{\Sigma^+(\mathcal Z,g_\mathcal{Z})}(\nu,Y)\varphi&=\frac{1}{2}\left(\frac{\beta''(t)}{\beta(t)}-\frac{1}{2}\frac{\alpha'(t)}{\beta(t)^2}a+\frac{1}{2}\frac{\alpha(t)\beta'(t)}{\beta(t)^3}a\right)\nu\cdot Y\cdot\varphi
\end{align*}
where $(\mathcal{Z},g_\mathcal{Z})$ is with respect to \eqref{eq ansatz for g_t}.
\end{lemma}
\begin{proof} We use \eqref{eq spingeo curvature}. As local orthonormal frame for $\mathcal{Z}$ we choose $(\nu,e_0(.), e_1(.), e_2(.))$. Furthermore, we use the notation
\[R_{X,Y,Z,W}:=g_{\mathcal{Z}}(R^\mathcal Z(X,Y)Z,W)\]
With the aid of Lemma \ref{lemma z1} and \eqref{eq cy 3}-\eqref{eq cy 4} it follows from straight forward calculations (for details we refer to \cite[Lemma 6.9]{JW}), that
\begin{align*}
R_{\nu,e_0,e_0,e_1}&=R_{\nu,e_0,e_0,e_2}=R_{\nu,e_0,\nu,e_1}=R_{\nu,e_0,\nu,e_2}=0,\\
R_{\nu,e_0,\nu,e_0}&=\frac{\alpha''(t)}{\alpha(t)},\\
R_{\nu,e_0,e_1,e_2}&=\frac{\alpha'(t)}{\beta(t)^2}a-\frac{\alpha(t)\beta'(t)}{\beta(t)^3}a.
\end{align*}
Plugging this into \eqref{eq spingeo curvature} and using \eqref{eq spingeo cliffordm onb}, we get
\begin{align*}
R^{\Sigma^+\mathcal{Z}}(\nu,e_0)\varphi&=\frac{1}{2}\left(R_{\nu,e_0,\nu,e_0}\nu\cdot e_0\cdot\varphi + R_{\nu,e_0,e_1,e_2}e_1\cdot e_2\cdot\varphi\right)\\
&=\frac{1}{2}\left(\frac{\alpha''(t)}{\alpha(t)}\nu\cdot e_0\cdot\varphi + \left(\frac{\alpha'(t)}{\beta(t)^2}a-\frac{\alpha(t)\beta'(t)}{\beta(t)^3}a\right)\underbrace{e_1\cdot e_2\cdot\varphi}_{=\nu\cdot e_0\cdot \varphi}\right)\\
&=\frac{1}{2}\left(\frac{\alpha''(t)}{\alpha(t)}+\frac{\alpha'(t)}{\beta(t)^2}a-\frac{\alpha(t)\beta'(t)}{\beta(t)^3}a\right)\nu\cdot e_0\cdot\varphi.
\end{align*}
From this, the first equation in Lemma \ref{lemma curv in cyl} directly follows. The second equation follows from
\[R^{\Sigma^+\mathcal{Z}}(\nu,e_i)\varphi=\frac{1}{2}\left(\frac{\beta''(t)}{\beta(t)}-\frac{1}{2}\frac{\alpha'(t)}{\beta(t)^2}a+\frac{1}{2}\frac{\alpha(t)\beta'(t)}{\beta(t)^3}a\right)\nu\cdot e_i\cdot\varphi,\]
$i=1,2$, which is shown with the same method as above.

\end{proof}

\subsection{$S^1$-invariant spinors}\label{s1invspinors}
For details concerning this section we refer to \cite{ABCollaps} and also \cite{Mor}. Define $(g_t)_{t\in I}$ by \eqref{eq ansatz for g_t}. The $S^1$-action on $S^3$ induces an $S^1$-action on $\textup{SO}(S^3,g_t)$ which lifts uniquely to an $S^1$-action on $\textup{Spin}(S^3,g_t)$ as follows: We use the fact that for every Riemannian metric on $S^3$  (respectively, $S^2$) there exists, up to equivalence of reductions, exactly one metric spin structure. Pulling back $\textup{Spin}(S^2,\beta(t)^2g_{FS})$ along the Hopf fibration and enlarging the structure group to $\textup{Spin}(3)$ we get the spin structure on $(S^3,g_t)$,
\[\textup{Spin}(S^3,g_t)=\pi^*(\textup{Spin}(S^2,\beta(t)^2g_{FS}))\times_{\textup{Spin}(2)}\textup{Spin}(3).\]
Now we define the $S^1$-action by 
\[[(x,\sigma),g]\cdot e^{is}:=[(x\cdot e^{is},\sigma),g]\]
for $(x,\sigma)\in\pi^*(\textup{Spin}(S^2,\beta(t)^2g_{FS}))\subset S^3\times \textup{Spin}(S^2,\beta(t)^2g_{FS})$, $g\in\textup{Spin}(3)$, and $e^{is}\in S^1\subset\mathbb{C}$. This action is the desired lift of the $S^1$-action on $\textup{SO}(S^3,g_t)$. Uniqueness follows from the fact that $\textup{Spin}(S^3,g_t)\cong S^3\times\textup{Spin}(3)$ and $\textup{SO}(S^3,g_t)\cong S^3\times SO(3,\mathbb{R})$ are connected.

This yields a $S^1$-action on $\Sigma(S^3,g_t)$. Spinors which are invariant under this action are called \textit{$S^1$-invariant}. Denote by $V(t)\subset\Gamma(\Sigma(S^3,g_t))$ the vector space of $S^1$-invariant spinors. For every $\varphi\in V(t)$ we have 
\begin{align}\label{eq s1 inv 1}
\nabla_K^t\varphi=\frac{1}{4}\frac{\alpha(t)}{\beta(t)^2}aK\cdot_t\varphi,
\end{align}
see \cite[Lemma 4.3.]{ABCollaps}.

The $S^1$-invariant spinors are in one-to-one correspondence to the spinors on the base manifold. To be more precise, by \cite[Lemma 4.4.]{ABCollaps} there is an isomorphism of vector spaces
\[Q=Q(t)\colon \Gamma(\Sigma(S^2,\beta(t)^2g_{FS}))\rightarrow V(t).\]
The following identities will be used later: For every vector field $X\in\Gamma(TS^2)$ and every spinor $\sigma\in\Gamma(\Sigma(S^2,\beta(t)^2g_{FS}))$
\begin{align}\label{eq s1 inv 2}
\nabla_{X^*}^tQ(\sigma)=Q(\nabla_X^t\sigma)-\frac{1}{4}\frac{\alpha(t)}{\beta(t)^2}aX^*\cdot_t Q(\sigma),
\end{align}
\begin{align}\label{eq s1 inv 3}
Q(X\cdot_t \sigma)=X^*\cdot_t Q(\sigma).
\end{align}
In \eqref{eq s1 inv 1}-\eqref{eq s1 inv 3} we denote by $\nabla^t$ the spinorial Levi-Civita connection on $\Sigma(S^3,g_t)$ and $\Sigma(S^2,\beta(t)^2g_{FS})$ respectively and ``$\cdot_t$'' is the clifford multiplication in the respective spinor bundles.

\subsection{A collapsing theorem}
Our first main result is the following theorem.
\begin{theorem}[Collapse]\label{main result 1} Let $\varepsilon>0$ and $\lambda\in\{\pm1\}$. Write $(g_0,\varphi_0):=(g^\varepsilon,Q(\sigma))$ for $\sigma$ a $\lambda$-Killing spinor on $(S^2,g_{FS}=\frac{1}{4}g_{S^2})$ such that $|\varphi_0|=1$. Then the solution of the spinor flow on $M=S^3$ with initial value $(g_0,\varphi_0)$ is given by \eqref{eq ansatz for g_t}-\eqref{eq ansatz for varphi_t} where $b=t_{max}$, $t_{max}\in(0,\infty]$ maximum time of existence (to the right), such that:
\begin{itemize}
\item[] If $a\lambda=2$, then:
	\begin{itemize}
	\item[$\bullet$] For $0<\varepsilon<\frac{2}{3}$ we have $t_{max}=\infty$, $\lim\limits_{t\rightarrow \infty}\alpha(t)=0$, and $\lim\limits_{t\rightarrow \infty}\beta(t)=:\beta_\infty>0$.
	\item[$\bullet$] For $\varepsilon\ge \frac{2}{3}$ we have $\lim\limits_{t\rightarrow t_{max}}\alpha(t)=\lim\limits_{t\rightarrow t_{max}}\beta(t)=0$.
	\item[$\bullet$] For $\varepsilon =\frac{2}{3}$ we have $t_{max}=12$, $\alpha(t)=\frac{2}{3}\beta(t)$, and  $\beta(t)=\frac{1}{6}\sqrt{36-3t}$.
	\item[$\bullet$] For $\varepsilon =1$ we have $t_{\max}=16$ and $\alpha(t)=\beta(t)=\frac{1}{4}\sqrt{16-t}$.
	\end{itemize}
\item[] If $a\lambda=-2$, then:
	\begin{itemize}
	\item[$\bullet$] For every $\varepsilon >0$ we have $t_{max}=\infty$, $\lim\limits_{t\rightarrow \infty}\alpha(t)=0$, and $\lim\limits_{t\rightarrow \infty}\beta(t)=:\beta_\infty>0$.
	\end{itemize}
\end{itemize}
Moreover, in any of the above cases the spinor flow preserves the class of $S^1$-invariant spinors which correspond to Killing spinors on $S^2$. More precisely,
\[\varphi_t=Q(\sigma_t)\]
for every $t\in I$ where $\sigma_t$ is a $\frac{\lambda}{\beta(t)}$-Killing spinor on $(S^2,\beta(t)^2g_{FS})$.
\end{theorem}

\begin{remark}
In the case $a\lambda=2$ the result can be interpreted as follows: If we start with fibers that are sufficiently short ($\varepsilon <\frac{2}{3}$), then the $S^1$-fiber converges to a point under the spinor flow ($\alpha\to 0$), but the complement does not ($\beta\to\beta_\infty>0$). In that sense the $S^1$-principal bundle $S^3$ collapses against its base $S^2$. If we start with fibers that are too long ($\varepsilon >\frac{2}{3}$), then $S^3$ converges to a point under the spinor flow.

In the case $a\lambda=-2$ the collapse is independent of the length of the fibers.
\end{remark}

Now we carry out the steps mentioned in the introduction to prove Theorem \ref{main result 1}.

\begin{lemma}\label{lemma eq for varphi_t}Choose $(g_0,\varphi_0)$ as in Theorem \ref{main result 1} and define $(g_t,\varphi_t)_{t\in I}$ by \eqref{eq ansatz for g_t}-\eqref{eq ansatz for varphi_t}. Then, for every $t\in I$ and every horizontal vector field $Y\in \Gamma(TS^3)$, it holds that
\begin{align*}
\nabla^t_K\varphi_t&=\frac{1}{4}\frac{\alpha(t)}{\beta(t)^2}aK\cdot_t\varphi_t,\\
\nabla^t_Y\varphi_t&=(\frac{1}{\beta(t)}\lambda-\frac{1}{4}\frac{\alpha(t)}{\beta(t)^2}a)Y\cdot_t\varphi_t
\end{align*} 
where $\nabla^t=\nabla^{\Sigma(S^3,g_t)}$ and  ``$\cdot_t$'' is the Clifford multiplication in $\Sigma(S^3,g_t)$.
\end{lemma} 
\begin{proof}
First of all, from \eqref{eq s1 inv 1}-\eqref{eq s1 inv 3} we get 
\begin{align*}
\nabla^0_K\varphi_0&=\frac{1}{4}\varepsilon aK\cdot_0\varphi_0,\\
\nabla^0_Y\varphi_0&=(\lambda-\frac{1}{4}\varepsilon a)Y\cdot_0\varphi_0,
\end{align*}
for every horizontal vector field $Y\in\Gamma(TS^3)$. From \eqref{eq cy 6} we get $\nabla^\mathcal{Z}_\nu e_0=0$ and \eqref{eq cy 5} yields $[\nu,e_0](t,p)=-\frac{\alpha'(t)}{\alpha(t)}e_0(t,x)$. In the following we write $\nabla^{\Sigma^+\mathcal Z}=\nabla^{\Sigma^+{(\mathcal{Z},g_\mathcal{Z})}}$. Using Lemma \ref{lemma cyl identi} and Lemma \ref{lemma curv in cyl} it follows that
\begin{align*}
\nabla&^{\Sigma^+\mathcal{Z}}_\nu\left(\nabla^{t}_{e_0(t)}\varphi_t -\frac{1}{4}\frac{\alpha(t)}{\beta(t)^2}ae_0(t)\cdot_t\varphi_t\right)\\ &=\nabla^{\Sigma^+\mathcal{Z}}_\nu\left(\nabla^{\Sigma^+\mathcal{Z}}_{e_0}\varphi+\frac{1}{2}\nu\cdot \underbrace{W_t(e_0(t))}_{=-\frac{\alpha'(t)}{\alpha(t)}e_0(t)}\cdot\varphi\right)- \frac{1}{4}a\nabla^{\Sigma^+\mathcal{Z}}_\nu\left(\frac{\alpha(t)}{\beta(t)^2}\nu\cdot e_0\cdot\varphi\right)\\
&=\nabla^{\Sigma^+\mathcal{Z}}_\nu\nabla^{\Sigma^+\mathcal{Z}}_{e_0}\varphi -\frac{1}{2}L_\nu(\frac{\alpha'(t)}{\alpha(t)})\nu\cdot e_0\cdot\varphi-\frac{1}{4}aL_\nu(\frac{\alpha(t)}{\beta(t)^2})\nu\cdot e_0\cdot \varphi\\
&=R^{\Sigma^+\mathcal{Z}}(\nu,e_0)\varphi + \nabla^{\Sigma^+\mathcal{Z}}_{e_0}\underbrace{\nabla^{\Sigma^+\mathcal{Z}}_\nu\varphi}_{=0} + \nabla^{\Sigma^+\mathcal{Z}}_{[\nu,e_0]}\varphi -\left(\frac{1}{2}L_\nu(\frac{\alpha'(t)}{\alpha(t)})+\frac{1}{4}aL_\nu(\frac{\alpha(t)}{\beta(t)^2})\right)\nu\cdot e_0\cdot \varphi\\
&=-\frac{\alpha'(t)}{\alpha(t)}\nabla^{\Sigma^+\mathcal{Z}}_{e_0}\varphi + R^{\Sigma^+\mathcal{Z}}(\nu,e_0)\varphi - \left(\frac{1}{2}(\frac{\alpha'(t)}{\alpha(t)})'+\frac{1}{4}a(\frac{\alpha(t)}{\beta(t)^2})'\right)\nu\cdot e_0\cdot\varphi\\
&=-\frac{\alpha'(t)}{\alpha(t)}\nabla^{\Sigma^+\mathcal{Z}}_{e_0}\varphi + \frac{1}{2}\left(\frac{\alpha''(t)}{\alpha(t)}+\frac{\alpha'(t)}{\beta(t)^2}a-\frac{\alpha(t)\beta'(t)}{\beta(t)^3}a\right)\nu\cdot e_0\cdot\varphi\\
&\hspace{5em}-\left(\frac{1}{2}\frac{\alpha''(t)}{\alpha(t)}-\frac{1}{2}\frac{\alpha'(t)^2}{\alpha(t)^2}+\frac{1}{4}a\frac{\alpha'(t)}{\beta(t)^2}-\frac{1}{2}a\frac{\alpha(t)\beta'(t)}{\beta(t)^3}\right) \nu\cdot e_0\cdot\varphi\\
&=-\frac{\alpha'(t)}{\alpha(t)}\nabla^{\Sigma^+\mathcal{Z}}_{e_0}\varphi + \left(\frac{1}{4}\frac{\alpha'(t)}{\beta(t)^2}a+\frac{1}{2}\frac{\alpha'(t)^2}{\alpha(t)^2}\right)\nu\cdot e_0\cdot\varphi\\
&=-\frac{\alpha'(t)}{\alpha(t)}\left(\nabla^t_{e_0(t)}\varphi_t+\frac{1}{2}\frac{\alpha'(t)}{\alpha(t)}\nu\cdot e_0\cdot\varphi\right)+ \left(\frac{1}{4}\frac{\alpha'(t)}{\beta(t)^2}a+\frac{1}{2}\frac{\alpha'(t)^2}{\alpha(t)^2}\right)\nu\cdot e_0\cdot\varphi\\
&=-\frac{\alpha'(t)}{\alpha(t)}\left(\nabla^{t}_{e_0(t)}\varphi_t -\frac{1}{4}\frac{\alpha(t)}{\beta(t)^2}ae_0(t)\cdot_t\varphi_t\right).
\end{align*}
We have shown
	\begin{center}
		$ \left\{
		\begin{array}{lll}
		\nabla^{\Sigma^+\mathcal{Z}}_\nu&\left(\nabla^{t}_{e_0(t)}\varphi_t -\frac{1}{4}\frac{\alpha(t)}{\beta(t)^2}ae_0(t)\cdot_t\varphi_t\right)&=-\frac{\alpha'(t)}{\alpha(t)}\left(\nabla^{t}_{e_0(t)}\varphi_t -\frac{1}{4}\frac{\alpha(t)}{\beta(t)^2}ae_0(t)\cdot_t\varphi_t\right),\\
		&\nabla^{0}_{e_0(0)}\varphi_0-\frac{1}{4}\frac{\alpha(0)}{\beta(0)^2} ae_0(0)\cdot_0\varphi_0&=0.
		\end{array}
		\right.$
	\end{center}
This differential equation with initial value has zero as unique solution, so
\[\nabla^{t}_{e_0(t)}\varphi_t -\frac{1}{4}\frac{\alpha(t)}{\beta(t)^2}ae_0(t)\cdot_t\varphi_t=0\]
on $\mathcal{Z}$.

To prove the second equation, we define $\tilde{Y}_t\in \Gamma(TS^3)$ by $\tilde{Y}_t(x):=\frac{1}{\beta(t)}Y(x)$. Using the same ideas as above, we get
\[\nabla^t_{\widetilde{Y}_t}\varphi_t-\left(\frac{1}{\beta(t)}\lambda-\frac{1}{4}\frac{\alpha(t)}{\beta(t)^2}a\right)\widetilde{Y}_t\cdot_t\varphi_t=0.\]
\end{proof}

Using the previous lemma, a straightforward calculation yields the following lemma. (Details can be found in \cite[Lemma 6.14]{JW}.)
\begin{lemma}\label{lemma q-terms 1}Let $(g_t,\varphi_t)_{t\in I}$ as in Lemma \ref{lemma eq for varphi_t}. For every $t\in I$ it holds that
	\begin{align*}
		Q_1(g_t,\varphi_t)(e_0(t),e_0(t))&=-\frac{9}{64}\frac{\alpha(t)^2}{\beta(t)^4}a^2+\frac{1}{2}\frac{\alpha(t)}{\beta(t)^3}a\lambda -\frac{1}{2}\frac{1}{\beta(t)^2}\lambda^2,\\
		Q_1(g_t,\varphi_t)(e_1(t),e_1(t))&=\frac{3}{64}\frac{\alpha(t)^2}{\beta(t)^4}a^2-\frac{1}{8}\frac{\alpha(t)}{\beta(t)^3}a\lambda,\\
		Q_1(g_t,\varphi_t)(e_2(t),e_2(t))&=Q_1(g_t,\varphi_t)(e_1(t),e_1(t)),\\
		Q_1(g_t,\varphi_t)(e_i(t),e_j(t))&=0 \text{ for }i\neq j,\\
		Q_2(g_t,\varphi_t)&=0.
	\end{align*}
\end{lemma}

\begin{proof}[Proof of Theorem \ref{main result 1}]
First of all, we have 
\[\frac{\partial}{\partial t}\varphi_t=\nabla^{\Sigma^+(\mathcal{Z},g_\mathcal{Z})}_\frac{\partial}{\partial t}\varphi =0=Q_2(g_t,\varphi_t).\]	
Moreover, $|\varphi_t|=1$ for all $t\in I$ follows from the fact that $\nabla^{\Sigma^+(\mathcal{Z},g_\mathcal{Z})}$ is a metric connection. From Lemma \ref{lemma q-terms 1} we deduce that $\frac{\partial}{\partial t}g_t=Q_1(g_t,\varphi_t)$ with $g_0=g^\varepsilon$ holds iff $(\alpha,\beta)$ is the solution of the following system of two non-linear ordinary differential equations:
\begin{align*}
\alpha'(t)&=-\frac{9}{128}\frac{\alpha(t)^3}{\beta(t)^4}a^2+\frac{1}{4}\frac{\alpha(t)^2}{\beta(t)^3}a\lambda -\frac{1}{4}\frac{\alpha(t)}{\beta(t)^2}\lambda^2,\\
\beta'(t)&=\frac{3}{128}\frac{\alpha(t)^2}{\beta(t)^3}a^2-\frac{1}{16}\frac{\alpha(t)}{\beta(t)^2}a\lambda,\\
\alpha(0)&=\varepsilon,\\
\beta(0)&=1.
\end{align*}
Let $F\colon U:=\mathbb{R}^2\setminus\{(x,y)\in\mathbb{R}^2|\text{ }y=0\}\rightarrow \mathbb{R}^2$ be the vector field associated to that system, i.e.
\[F\begin{pmatrix}
x \\ y
\end{pmatrix}=\begin{pmatrix}
-\frac{9}{128}\frac{x^3}{y^4}a^2+\frac{1}{4}\frac{x^2}{y^3}a\lambda -\frac{1}{4}\frac{x}{y^2}\lambda^2 \\ 
\frac{3}{128}\frac{x^2}{y^3}a^2-\frac{1}{16}\frac{x}{y^2}a\lambda
\end{pmatrix},\]
see figure \ref{fig plot of F from main result 1}.

\begin{figure}[]
	\subfigure[$F$ for $a\lambda=2$ with integral curves. The integral curves with starting point $(1,1)$ and $(\frac{2}{3},1)$ are highlighted.]{\includegraphics[width=0.40\textwidth]{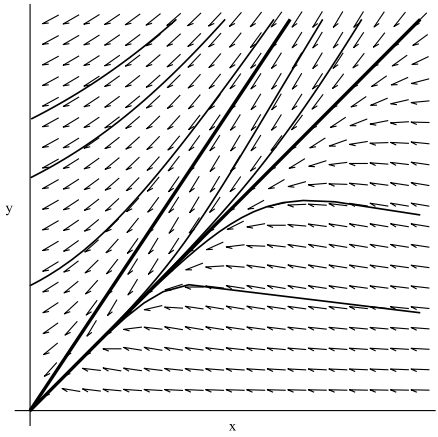}}\hfill
	\subfigure[$F$ for $a\lambda=-2$ with integral curves]{\includegraphics[width=0.40\textwidth]{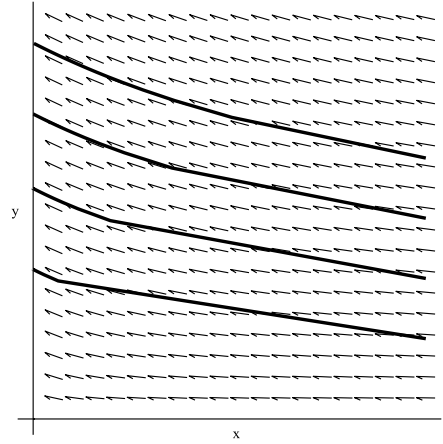}}\\
	\caption{Plots of the vector field $F$ of Theorem \ref{main result 1}.}
	\label{fig plot of F from main result 1}
\end{figure}
Let $c=(x,y)\colon J\rightarrow U$, $J\subset\mathbb{R}$ interval, be an integral curve of $F$. If there exists $t\in J$ such that $x(t)\neq 0$, then $c(J)$ lies in one quadrant of $\mathbb{R}^2$. Using
\[x'(t)=\frac{x(t)}{y(t)^2}\left(-\frac{9}{128}\left(\frac{x(t)}{y(t)}a\lambda\right)^2+\frac{1}{4}\left(\frac{x(t)}{y(t)}a\lambda\right) -\frac{1}{4}\right)\]
and $-\frac{9}{128}z^2+\frac{1}{4}z -\frac{1}{4}<0$ for all $z\in\mathbb{R}$, we get that $x(t)$ is either strictly decreasing or strictly increasing (depending on in which quadrant $c(J)$ lies). It follows that the critical points of $F$ are precisely the points $(0,k)$ for $k\neq0$.

Let us now prove the case $a\lambda=-2$. First we show that the integral curves of $F$ remain in certain compact subsets of $U$. To that end, let
\[K(v,w):=\{(x,y)\in\mathbb{R}^2\text{ }|\text{ }0\le x\le v, \text{ } w\le y \le w+v-x\}.\]
If $c$ is an integral curve of $F$ as above with $x(l),y(l)>0$ for some $l\in J$, then $c(t)\in K(c(l))$ for all $t\in J$ with $t\ge l$. The idea to prove that is as follows: For every boundary point $(v,w)\in \partial K(c(l))$ with $v>0$ the vector $F(v,w)$ points inside $K(c(l))$. Then the integral curve $c$ can't leave $K(c(l))$ since its movement is prescribed by $F$.

Let $c\colon J\rightarrow U$ be a maximal integral curve of $F$ with $c(0)=(\varepsilon,1)$, $\varepsilon>0$, and $J=[0,t_{max})$. Then $c(J)\subset K(c(0))\subset U$. Therefore, we have $t_{max}=\infty$. Using the Poincaré-Bendixson theorem (we use the version from \cite{CodLevODE}) and observing that $F$ has no periodic orbits, we get a sequence $(t_n)_{n\in \mathbb{N}}$ such that $t_n\ge 0$, $t_n\to \infty$,  and $c(t_n)\to p$ for $n\to \infty$ where $p$ is a critical point of $F$. So we have $p=(0,k)$ for some $k>0$. It follows that
\[\lim\limits_{t\to \infty}c(t)=(0,k),\]
since for every $n\in\mathbb{N}$ we have $c([t_n,\infty))\subset K(c(t_n))$. This proves Theorem \ref{main result 1} in the case $a\lambda =-2$.

The case $a\lambda=2$ can be treated with the same methods, i.e. by showing that integral curves remain in certain compact sets and using the Poincaré-Bendixson theorem. This time, however, we have to consider three different cases depending on the value of $\varepsilon >0$. We briefly outline the proof in this case. Define
\begin{align*}
	K_1(v,w)&:=\{(x,y)\in\mathbb{R}^2\text{ }|\text{ }0\le x\le v,\text{ }\frac32x+w-\frac32v\le y\le w\} \text{ for } 0<v<\frac23 w,\\
	K_2(v)&:=\{(x,y)\in\mathbb{R}^2\text{ }|\text{ }0\le x\le v,\text{ } x\le y\le\frac32 x\}\text{ for } v>0,\\
	K_3(v,w)&:=\{(x,y)\in\mathbb{R}^2\text{ }|\text{ }0\le x\le v, \frac{w}{v}x\le y\le x\}\text{ for } 0<w<v,\\
\end{align*}
see figure \ref{K123}.

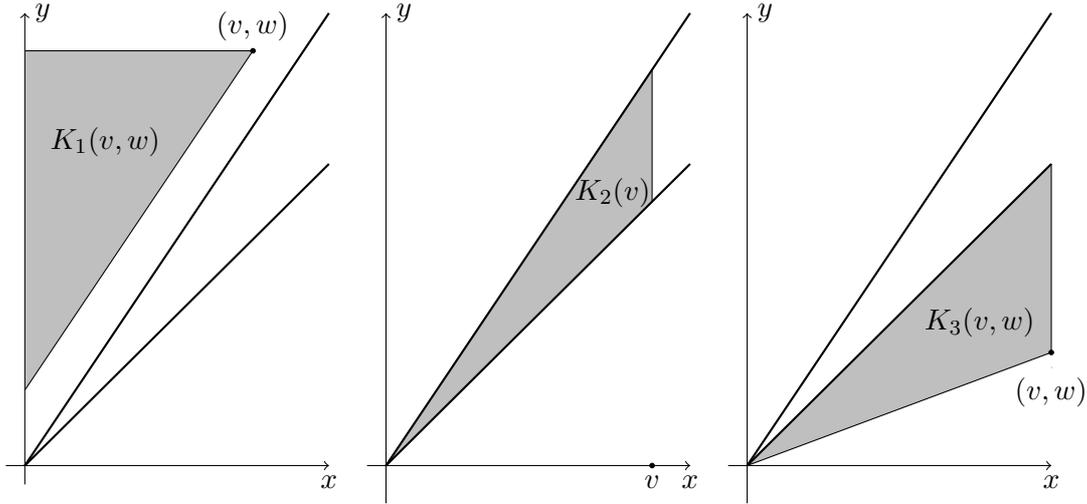
\begin{figure}\mbox{}\hfill
\begin{tikzpicture}
\path[fill=lightgray] (-0.25,1) .. controls (1.75,4).. (2.75,5.5) .. controls (1.75,5.5).. (-0.25,5.5) .. controls (-0.25,3).. (-0.25,1);
\draw [->] (-0.5,0)-- (3.75,0) node[below] {$x$};
\draw [->] (-0.25,-0.25)--(-0.25,6)node[right]{$y$};
\draw[thick] (-0.25,0) -- (3.75,6);
\draw[thick] (-0.25,0) -- (3.75,4);
\draw (-0.25,1) -- (2.75,5.5);
\draw (-0.25,5.5) -- (2.75,5.5);
\draw [fill=black] (2.75,5.5) circle (0.03)node[above] {$(v,w)$};
\draw (-0.05,4.3) node[right] {$K_1(v,w)$};

\draw [->] (4.25,0)-- (8.5,0) node[below] {$x$};
\draw [->] (4.5,-0.5)--(4.5,6)node[right]{$y$};
\path[fill=lightgray] (4.5,0) .. controls (5.5,1).. (8,3.5) .. controls (8,4).. (8,10.5/2) .. controls (5.5,1.5).. (4.5,0);
\draw[thick] (4.5,0) -- (8.5,6);
\draw[thick] (4.5,0) -- (8.5,4);
\draw (8,3.5) -- (8,10.5/2);
\draw (6.85,3.62) node[right] {$K_2(v)$};
\draw [fill=black] (8,0) circle (0.03)node[below] {$v$};

\draw [->] (9,0)-- (13.25,0) node[below] {$x$};
\draw [->] (9.25,-0.5)--(9.25,6)node[right]{$y$};
\path[fill=lightgray] (9.25,0) .. controls (11.25,6/8).. (13.25,1.5) .. controls (13.25,2).. (13.25,4) .. controls (12.25,3).. (9.25,0);
\draw[thick] (9.25,0) -- (13.25,6);
\draw[thick] (9.25,0) -- (13.25,4);
\draw (9.25,0)-- (13.25,1.5);
\draw (13.25,1.5) -- (13.25,4);
\draw (11.45,1.9) node[right] {$K_3(v,w)$};
\draw [fill=black] (13.25,1.5) circle (0.03);
\draw [fill=black] (13.25,1.3) circle (0.00)node[below] {$(v,w)$};
\end{tikzpicture}
	\hfill\mbox{}
	\caption{The sets $K_1(v,w)$, $K_2(v)$, and $K_3(v,w)$ together with the integral curves of $F$ through $(1,1)$ and $(\frac23,1)$ for $a\lambda=2$.}
	\label{K123}
\end{figure}

\textbf{Proof for $a\lambda=2$ and $0<\varepsilon<\frac{2}{3}$:} We show as before: If $c=(x,y)\colon J \rightarrow U$ is an integral curve of $F$ with $0<x(l)<\frac{2}{3}y(l)$ and $y(l)>0$ for some $l\in J$, then $c(t)\in K_1(c(l))$ for all $t\in J$ with $t\ge l$. Using the Poincaré-Bendixson theorem as above proves the theorem in this case.

\textbf{Proof for $a\lambda=2$ and $\frac{2}{3}<\varepsilon<1$:} Again we have: If $c=(x,y)\colon J \rightarrow U$ is an integral curve of $F$ with $0<x(l)<y(l)<\frac{3}{2}x(l)$ for some $l\in J$, then $c(t)\in K_2(x(l))$ for all $t\in J$ with $t\ge l$. Let $c=(x,y)\colon J\rightarrow U$ be a maximal integral curve of $F$ with $c(0)=(\varepsilon,1)$, $\frac{2}{3}<\varepsilon<1$, and $J=[0,t_{max})$. We show $c(t)\xrightarrow{t\to t_{max}}0$. It holds that $c(J)\subset K_2(x(0))$. Because of the Poincaré-Bendixson theorem, there exists no $\delta>0$ such that $c(J)\subset K_2(x(0))\cap \{(x,y)\in\mathbb{R}^2\text{ }|\text{ }x\ge\delta\}$. Together with the fact that $x(t)$ is strictly decreasing we get $x(t)\xrightarrow{t\to t_{max}}0$ and therefore $c(t)\xrightarrow{t\to t_{max}}0$. This completes the proof in that case.

\textbf{Proof for $a\lambda=2$ and $\varepsilon>1$:} If $c=(x,y)\colon J \rightarrow U$ is an integral curve of $F$ with $0<y(l)<x(l)$ for some $l\in J$, then $c(t)\in K_3(c(l))$ for all $t\in J$ with $t\ge l$. Now proceed as in the case $\frac{2}{3}<\varepsilon<1$.

It remains to prove the last statement of the theorem. Let $\widetilde{\textup{GL}}^+S^3$ be the topological spin structure on $S^3$. For fixed $e^{is}\in S^1$ we construct a spin-diffeomorphism $F\colon \widetilde{\textup{GL}}^+S^3\rightarrow \widetilde{\textup{GL}}^+S^3$ which restricts to the action of $e^{is}$ on $\textup{Spin}(S^3,g_t)$ defined in Section \ref{s1invspinors}, for every $t\in I$. (We use the definition of ``spin-diffeomorphism'' which is given in \cite[Section 4.1]{AWWsflow}.) Then we use the diffeomorphism invariance of the spinor flow (see \cite[Corollary 4.5. (ii)]{AWWsflow}) together with the uniqueness of the solution finish the proof.

The $S^1$-action on $S^3$ induces an $S^1$-action on $\textup{GL}^+S^3$ which lifts uniquely to a $S^1$-action on $\widetilde{\textup{GL}}^+S^3$. This can be shown as in the case of metric spin structures (see Section \ref{s1invspinors}) using topological spin structures instead. From the action of $e^{is}$ on $S^3$, $\widetilde{\textup{GL}}^+S^3$ and on $\textup{Spin}(S^3,g_t)$ (the latter is defined in Section \ref{s1invspinors}) we then get maps
\[F\colon \widetilde{\textup{GL}}^+S^3\rightarrow \widetilde{\textup{GL}}^+S^3,\]
\[F_t\colon \textup{Spin}(S^3,g_t)\rightarrow \textup{Spin}(S^3,g_t),\]
\[f\colon S^3\rightarrow S^3\]
where $F$ is a spin-diffeomorphism.
The $S^1$-actions on $\textup{GL}^+S^3$ and $\textup{SO}(S^3,g_t)$ coincide on $\textup{SO}(S^3,g_t)\subset\textup{GL}^+S^3$. Combining that with the uniqueness of the $S^1$-action on $\textup{Spin}(S^3,g_t)$ we get
\[F|_{\textup{Spin}(S^3,g_t)}=F_t\]
for every $t\in I$. Using \cite[Section 4.1]{AWWsflow} we get a map
\[F_*\colon \mathcal{N}\rightarrow\mathcal{N},\hspace{3em} \mathcal{N}_g\ni\varphi\mapsto F_*\varphi\in\mathcal{N}_{(f^{-1})^*g},\]
defined by: If locally $\varphi=[\tilde{s},\tilde{\varphi}]$, then $F_*\varphi=[F\circ\tilde{s}\circ f^{-1},\tilde{\varphi}\circ f^{-1}]$. From the definitions it follows that for every $t\in I$ and every spinor $\varphi\in\mathcal{N}_{g_t}$
\begin{align}\label{eq z1}
\left(F_*\varphi\right)(x)=\varphi(x\cdot e^{-is})\cdot e^{is}.
\end{align}
Now let $(g_t,\varphi_t)_{t\in I}$ be the solution of the spinor flow with initial value as in Theorem \ref{main result 1}. By the diffeomorphism invariance of the spinor flow, $((f^{-1})^*g_t,F_*\varphi_t)_{t\in I}$ is also a solution. We have $(f^{-1})^*g_0=(f^{-1})^*g^\varepsilon=g^\varepsilon$ and from \eqref{eq z1} we get $F_*\varphi_0=\varphi_0$. Because of the uniqueness of the solution of the spinor flow it follows that
\[F_*\varphi_t=\varphi_t\]
for every $t\in I$. Using again \eqref{eq z1} and noting that $e^{is}\in S^1$ was arbitrary, we see that $\varphi_t$ is $S^1$-invariant, i.e. $\varphi_t\in V(t)$. Define $\sigma_t\in\Gamma(\Sigma(S^2,\beta(t)^2g_{FS}))$ by $Q(\sigma_t)=\varphi_t$. Combining the second equation in Lemma \ref{lemma eq for varphi_t} with \eqref{eq s1 inv 2} and \eqref{eq s1 inv 3} yields that $\sigma_t$ is a $\frac{\lambda}{\beta(t)}$-Killing spinor. This finishes the proof of the theorem.
\end{proof}

\begin{remark}[Change of orientation convention]\label{rem change or conv} If we choose the other orientation on $S^3$ (i.e. the orientation that satisfies: If $(f_1,f_2)$ is any  oriented local orthonormal frame on $(S^2,g_{FS})$, then $(-K,f_1^*,f_2^*)$ is an oriented local orthonormal frame on $(S^3,g_{S^3})$), then Theorem \ref{main result 1} still holds if we switch the results for the cases ``$a\lambda=2$'' and ``$a\lambda=-2$''. This can be seen as follows: In Lemma \ref{lemma curv in cyl} one has to replace ``$a$'' by ``$-a$''. The additional sign enters because in the proof we used \eqref{eq spingeo cliffordm onb}. For the same reason we have to replace ``$a$'' by ``$-a$'' in \eqref{eq s1 inv 1}-\eqref{eq s1 inv 2}. Then we also need to replace ``$a$'' by ``$-a$'' in Lemma \ref{lemma eq for varphi_t} and Lemma \ref{lemma q-terms 1} and therefore also in Theorem \ref{main result 1}.
\end{remark}

\subsection{A stability theorem}
For $\varepsilon>0$ we define $c(\varepsilon)>0$ by $\textup{vol}(S^3,c(\varepsilon)g^\varepsilon)=1$, i.e. $c(\varepsilon)=(\frac{1}{2\pi^2\varepsilon})^{\frac{2}{3}}$.

Our second main result is the following theorem.
\begin{theorem}(Stability)\label{main result 2} Let $\varepsilon>0$ and $\mu\in\{\pm\frac{1}{2}\}$. Moreover, let $g_0:=c(\varepsilon)g^\varepsilon$ and $\varphi_0$ a spinor which is obtained via parallel transport of a $\mu$-Killing spinor from $(S^3,g_{S^3})$ to $(S^3,g^\varepsilon)$ such that $|\varphi_0|= 1$. (In Remark \ref{rmk main res 2} we will explain how $\varphi_0$ is defined in a more formal way.) Then the solution of the normalized spinor flow on $M=S^3$ with initial value $(g_0,\varphi_0)$ is given by \eqref{eq ansatz for g_t}-\eqref{eq ansatz for varphi_t} where $b=t_{max}$, $t_{max}\in(0,\infty]$ maximum time of existence (to the right), such that:
	\begin{itemize}
		\item[] If $a\mu=1$, then:
		\begin{itemize}
			\item[$\bullet$]For $0<\varepsilon<\frac{2}{3}$ we have $\lim\limits_{t\rightarrow t_{max}}\alpha(t)=0$ and $\lim\limits_{t\rightarrow t_{max}}\beta(t)=\infty$.
			\item[$\bullet$]For $\varepsilon =\frac{2}{3}$ we have $t_{max}=\infty$, $\alpha(t)\equiv\frac{2}{3}\sqrt{c(\frac{2}{3})}$, and $\beta(t)\equiv\sqrt{c(\frac{2}{3})}$.
			In particular, $(c(\frac23)g^\frac23,\varphi_0)$ is a critical point of the normalized spinor flow.
			\item[$\bullet$]For $\varepsilon >\frac{2}{3}$ we have $t_{max}=\infty$ and $\lim\limits_{t\rightarrow \infty}\alpha(t)=\lim\limits_{t\rightarrow \infty}\beta(t)=\sqrt{c(1)}$. Moreover, there exist smooth functions $f,g\colon I\rightarrow (0,\infty)$ with $\lim\limits_{t\rightarrow \infty}f(t)=\lim\limits_{t\rightarrow \infty}g(t)=\frac{\mu}{\sqrt{c(1)}}$ and
			\[\nabla^t_K\varphi_t=f(t)K\cdot_t\varphi_t,\]
			\[\nabla^t_Y\varphi_t=g(t)Y\cdot_t\varphi_t,\]
			for every horizontal vector field $Y\in\Gamma(TS^3)$ and every $t\in I$.
			\item[$\bullet$]For $\varepsilon =1$ we have $t_{max}=\infty$ and $\alpha(t)\equiv\beta(t)\equiv\sqrt{c(1)}$. In particular, $(c(1)g_{S^3},\varphi_0)$ is a critical point of the normalized spinor flow.
		\end{itemize}
		\item[] If $a\mu=-1$, then:
		\begin{itemize}
			\item[$\bullet$]For every $\varepsilon >0$ we have $t_{max}=\infty$ and $\lim\limits_{t\rightarrow \infty}\alpha(t)=\lim\limits_{t\rightarrow \infty}\beta(t)=\sqrt{c(1)}$. Moreover, there exist smooth functions $f,g\colon I\rightarrow (0,\infty)$ with $\lim\limits_{t\rightarrow \infty}f(t)=\lim\limits_{t\rightarrow \infty}g(t)=\frac{\mu}{\sqrt{c(1)}}$ and
			\[\nabla^t_K\varphi_t=f(t)K\cdot_t\varphi_t,\]
			\[\nabla^t_Y\varphi_t=g(t)Y\cdot_t\varphi_t,\]
			for every horizontal vector field $Y\in\Gamma(TS^3)$ and every $t\in I$.
			\item[$\bullet$]For $\varepsilon =1$ we have $t_{max}=\infty$ and $\alpha(t)\equiv\beta(t)\equiv\sqrt{c(1)}$. In particular, $(c(1)g_{S^3},\varphi_0)$ is a critical point of the normalized spinor flow.
		\end{itemize}
	\end{itemize}
	
\end{theorem}

\begin{remark}\label{rmk main res 2}\
	\begin{enumerate}
		\item In the case $a\mu=1$ we can interpret the result as follows: If we are not too far away ($\varepsilon >\frac23$) from the normalized standard metric $c(1)g_{S^3}$ together with a Killing spinor, then the metric part of the solution flows back to the normalized standard metric ($\alpha,\beta\to\sqrt{c(1)}$) and the spinor part of the solution flows back to a Killing spinor ($f,g\to\frac{\mu}{\sqrt{c(1)}}$). However, if we are too far away ($\varepsilon \le \frac23$), then the solution no longer flows back ($\alpha\to 0$, $\beta\to \infty$).
		\item If we choose $g_0=g^\varepsilon$ and $\varphi_0$ as in Theorem \ref{main result 2}, then $(g_0,\varphi_0)$ converges to a point under the unnormalized spinor flow (for $a\mu=-1$ or $a\mu=1$ and $\varepsilon\ge\frac{2}{3}$), see \cite[Theorem 6.17]{JW}. In that sense the interesting behavior is only captured in the normalized flow.
		\item In the following we make precise what we mean by $\varphi_0$ in Theorem \ref{main result 2}. To that end, let $\sigma_0\in\Gamma(\Sigma(S^3,g_{S^3}))$ be a $\mu$-Killing spinor with $|\sigma_0|=1$.
		
		\textbf{Case $\varepsilon >1$:} Define $\sigma_t$ by $\eqref{eq ansatz for varphi_t}$ where $\alpha(t)=1+t$, $\beta(t)\equiv 1$ and $b=\varepsilon -1$. Set \[\varphi_0:=\sigma_{\varepsilon-1}\in\Gamma(\Sigma(S^3,g^\varepsilon)).\]
		\textbf{Case $\varepsilon <1$:} Define $\sigma_t$ by $\eqref{eq ansatz for varphi_t}$ where $\alpha(t)=1-t$, $\beta(t)\equiv 1$, and $b=1-\varepsilon$. Set
		\[\varphi_0:=\sigma_{1-\varepsilon}\in\Gamma(\Sigma(S^3,g^\varepsilon)).\]
		\textbf{Case $\varepsilon =1$:} Simply set $\varphi_0:=\sigma_0$.
		
		Note that $\Sigma(S^3,g^\varepsilon)=\Sigma(S^3,c(\varepsilon)g^\varepsilon)$, so $\varphi_0\in\Gamma(\Sigma(S^3,c(\varepsilon)g^\varepsilon))$.
	\end{enumerate}
\end{remark}

\begin{lemma}\label{lemma eq for varphi_t 2} Choose $(g_0,\varphi_0)$ as in Theorem \ref{main result 2} and define $(g_t,\varphi_t)_{t\in I}$ by \eqref{eq ansatz for g_t}-\eqref{eq ansatz for varphi_t}. Then, for every $t\in I$ and every horizontal vector field $Y\in \Gamma(TS^3)$, it holds that
	\begin{align*}
	\nabla^t_K\varphi_t&=\left(\frac{\mu-\frac{1}{4}a}{\alpha(t)}+\frac{1}{4}\frac{\alpha(t)}{\beta(t)^2}a\right)K\cdot_t\varphi_t,\\
	\nabla^t_Y\varphi_t&=\frac{1}{\beta(t)}\left(-\frac{1}{4}a\left(\frac{\alpha(t)}{\beta(t)}-1\right)+\mu\right)Y\cdot_t\varphi_t.
	\end{align*} 
\end{lemma}
\begin{proof} First, these equations hold for $t=0$. (This can be shown as follows: Let $\sigma_t$ be as in Remark \ref{rmk main res 2}. We can deduce equations for $\nabla^t\sigma_t$ similar to the proof of Lemma \ref{lemma eq for varphi_t}. Evaluating the equations for $\nabla^t\sigma_t$ at $t=\varepsilon -1$ and $t=1-\varepsilon$, respectively, yields the equations of Lemma \ref{lemma eq for varphi_t 2} for $t=0$, see also \cite[Lemma 6.21]{JW}.) Using the same method as in the proof of Lemma \ref{lemma eq for varphi_t} yields the equations for all $t\in I$.
\end{proof}

Using Lemma \ref{lemma eq for varphi_t 2} one easily proves the following lemma.
\begin{lemma}\label{lemma q-terms 2}Let $(g_t,\varphi_t)_{t\in I}$ as in Lemma \ref{lemma eq for varphi_t 2}. For every $t\in I$ it holds, that
	\begin{align*}
	Q_1(g_t,\varphi_t)(e_0(t),e_0(t))&=\frac{1}{4}\frac{1}{\alpha(t)^2}\left(\mu-\frac{1}{4}a\right)^2+\frac{1}{\beta(t)^2}\left(-\frac{1}{2}\mu^2-\frac{3}{8}a\mu\right)\\
	&\hspace{2em}+\frac{\alpha(t)}{\beta(t)^3}\left(\frac{1}{8}a^2+\frac{1}{2}a\mu\right)-\frac{9}{64}\frac{\alpha(t)^2}{\beta(t)^4}a^2,\\
	Q_1(g_t,\varphi_t)(e_1(t),e_1(t))&=-\frac{1}{4}\frac{1}{\alpha(t)^2}\left(\mu-\frac{1}{4}a\right)^2+\frac{\alpha(t)}{\beta(t)^3}\left(-\frac{1}{32}a^2-\frac{1}{8}a\mu\right)+\frac{3}{64}\frac{\alpha(t)^2}{\beta(t)^4}a^2,\\
	Q_1(g_t,\varphi_t)(e_2(t),e_2(t))&=Q_1(g_t,\varphi_t)(e_1(t),e_1(t)),\\
	Q_1(g_t,\varphi_t)(e_i(t),e_j(t))&=0 \text{ for }i\neq j,\\
	Q_2(g_t,\varphi_t)&=0,\\
	\frac16 \frac{1}{\textup{vol}(S^3,g_t)}\mathcal{E}(g_t,\varphi_t)&=\frac{1}{12}(\mu-\frac{1}{4}a)^2\frac{1}{\alpha(t)^2}+ \frac{1}{64}a^2\frac{\alpha(t)^2}{\beta(t)^4}+\frac{1}{24}a(\mu-\frac{1}{4}a)\frac{1}{\beta(t)^2}\\
	&\hphantom{=}+\frac{1}{6}(\frac{1}{4}a+\mu)^2\frac{1}{\beta(t)^2}-\frac{1}{12}a(\frac{1}{4}a+\mu)\frac{\alpha(t)}{\beta(t)^3}.
	\end{align*}
\end{lemma}

\begin{proof}[Proof of Theorem \ref{main result 2}]
	From Lemma \ref{lemma q-terms 2} we get that $\frac{\partial}{\partial t}g_t=\tilde{Q}_1(g_t,\varphi_t)$ with $g_0=c(\varepsilon)g^\varepsilon$ is equivalent to the following systems of two non-linear ordinary differential equations for $(\alpha,\beta)$ with initial values	$\alpha(0)=\sqrt{c(\varepsilon)}\varepsilon$, $\beta(0)=\sqrt{c(\varepsilon)}$:
	For $a \mu=1$:
	\begin{align*}
	\alpha'(t)&=-\frac{1}{4}\frac{\alpha(t)^3}{\beta(t)^4}+\frac{5}{12}\frac{\alpha(t)^2}{\beta(t)^3}-\frac{1}{6}\frac{\alpha(t)}{\beta(t)^2},\\
	\beta'(t)&=\frac{1}{8}\frac{\alpha(t)^2}{\beta(t)^3}-\frac{5}{24}\frac{\alpha(t)}{\beta(t)^2}+\frac{1}{12}\frac{1}{\beta(t)},\\
	\end{align*}
	and for $a\mu=-1$:
	\begin{align*}
	\alpha'(t)&=-\frac{1}{4}\frac{\alpha(t)^3}{\beta(t)^4}+\frac{1}{12}\frac{\alpha(t)}{\beta(t)^2}+\frac{1}{6}\frac{1}{\alpha(t)},\\
	\beta'(t)&=\frac{1}{8}\frac{\alpha(t)^2}{\beta(t)^3}-\frac{1}{12}\frac{\beta(t)}{\alpha(t)^2}-\frac{1}{24}\frac{1}{\beta(t)}.
	\end{align*}
Denote by $F=F(x,y)$ the vector field associated to these differential equations as in the proof of Theorem \ref{main result 1}. First, we show that it suffices to restrict $F$ to a certain 1-dimensional submanifold of $\mathbb{R}^2$. To that end, define
\[u\colon (0,\infty)\rightarrow \mathbb{R}^2,\hspace{3em} t\mapsto \left( \sqrt{c(t)}t,\sqrt{c(t)} \right)=\left((2\pi^2)^{-\frac{1}{3}}t^{\frac{2}{3}},(2\pi^2)^{-\frac{1}{3}}t^{-\frac{1}{3}}\right).\]
The image of $u$ is an embedded submanifold of $\mathbb{R}^2$ and precisely the set of the initial values we are interested in. Noting that
\[F(u(t))=k(t)u'(t),\]
with
\[
k(t)=\left\{\begin{array}{ll} \frac{1}{8}(2\pi^2)^{\frac{2}{3}}t^{\frac{5}{3}}(-3t^2+5t-2), & \text{for }a\mu=1, \\
\frac{1}{8}(2\pi^2)^{\frac{2}{3}}t^{-\frac{1}{3}}(-3t^4+t^2+2), & \text{for }a\mu=-1,\end{array}\right. 
\] 	
we see that $F$ is tangent to the image of $u$. So we only have to understand the integral curves of $F|_{\textup{Im}(u)}$. One way to do this is to look at the corresponding vector field on $(0,\infty)$ which is given by $d(u^{-1})(V\circ u)$. Since $d(u^{-1})(V\circ u)=k(t)$ the claims about $\alpha$, $\beta$ and $t_{max}$ of Theorem \ref{main result 2} follow easily. The claimed convergence of the functions $f$ and $g$ follows from Lemma \ref{lemma eq for varphi_t 2} and the shown convergence of $\alpha$ and $\beta$. This finishes the proof of the theorem.	
\end{proof}

\newpage
\bibliographystyle{abbrv}
\bibliography{Literatur}

\end{document}